\documentclass[a4paper, 11pt, reqno]{amsart}

\usepackage{amsthm,amssymb,amstext,amscd,amsfonts,amsbsy,amsrefs,amsxtra,latexsym,amsmath,xcolor,mathrsfs,fancybox,upgreek,soul,url}
\usepackage[english]{babel}
\usepackage[all,cmtip]{xy}
\usepackage[latin1]{inputenc}
\usepackage{hyperref}
\usepackage{comment}
\usepackage{mdframed}
\allowdisplaybreaks
\usepackage{mathtools}
\usepackage{enumerate}
\usepackage{longtable}
\usepackage{thmtools}
\usepackage{thm-restate}
\usepackage{chngcntr}
\usepackage{enumitem}
\usepackage{etoolbox}
\usepackage{todonotes}
\usepackage{nicefrac}
\usepackage{float}

\usepackage{graphicx}

\usepackage{tikz}
\usepackage{verbatim}
\usetikzlibrary{quotes,angles}

\DeclarePairedDelimiter\abs{\lvert}{\rvert}
\DeclarePairedDelimiter\norm{\lVert}{\rVert}

\makeatletter
\let\oldabs\abs
\def\abs{\@ifstar{\oldabs}{\oldabs*}}
\let\oldnorm\norm
\def\norm{\@ifstar{\oldnorm}{\oldnorm*}}
\makeatother

\makeatletter
\glb@settings 
\makeatother

\newcommand{\leg}[2]{\genfrac{(}{)}{}{}{#1}{#2}}

\newtheorem{theorem}{Theorem}
\newtheorem{lemma}[theorem]{Lemma}
\newtheorem{corollary}[theorem]{Corollary}
\newtheorem{proposition}[theorem]{Proposition}

\theoremstyle{definition}
\newtheorem{definition}[theorem]{Definition}

\theoremstyle{remark}

\newtheorem*{remark}{Remark}

\newtheorem*{example}{Example}
\numberwithin{theorem}{section}
\numberwithin{proposition}{section}
\numberwithin{lemma}{section}
\numberwithin{corollary}{section}
\numberwithin{equation}{section}
\numberwithin{conjecture}{section}
\numberwithin{definition}{section}

\setlist[enumerate,1]{before=}
\AfterEndEnvironment{enumerate}{}

\newcommand{\N}{\mathbb{N}}
\newcommand{\Z}{\mathbb{Z}}
\newcommand{\R}{\mathbb{R}}
\newcommand{\C}{\mathbb{C}}
\newcommand{\Q}{\mathbb{Q}}

\newcommand{\HH}{\mathbb{H}}

\usepackage{bm}

\author{Walter Bridges, Johann Franke, and Johann Stumpenhusen}
\title{On the Proportion of Coprime Fractions in Number Fields}
\address{Department of Mathematics and Computer Science, Division of Mathematics, University of Cologne, Weyertal 86-90, 50931 Cologne, Germany}
\email{wbridges@uni-koeln.de}
\email{jfrank12@uni-koeln.de}
\email{jstumpen@math.uni-koeln.de}


\begin{document}

\maketitle

\begin{abstract}
    We determine the asymptotic density of coprime fractions in those of the reduced fractions of number fields. When ordered by norms of denominators, we count a fraction as soon as it ``appears'' for the first time and no later.  The natural density of coprime fractions in the set of reduced fractions may then be computed using well-known facts about Hecke $L$-functions. 
 Furthermore, we draw some connections to the modular group and Heegner points.
\end{abstract}

\section{Introduction}

\subsection{Motivation.} In the field of fractions of an integral domain, it is natural to pick a unique representative for each equivalence class of tuples that represent the same fraction. In a unique factorization domain (UFD), the following definition provides the typical characteristic which such representatives are required to fulfill.

\begin{definition}
    Let $R$ be an integral domain and $K$ its field of fractions, $a \in R$ and $b \in R \setminus \{0\}$. We call the fraction $\frac{a}{b}$ {\itshape reduced} if there exists no $c \in R \setminus R^\times$ such that $c \mid a$ and $c \mid b$.
\end{definition}

In the rational numbers, but also in all other number fields with class number 1, the terms reduced and coprime coincide, in the sense that a fraction is reduced if and only if numerator and denominator are coprime. However, this is no longer true for number fields with higher class number, so not every reduced fraction is already coprime. For example, the fraction
\[\frac{1 + \sqrt{-5}}{2}\]
is reduced in the field $\Q(\sqrt{-5})$ but not coprime since the ideal $\langle 2, 1 + \sqrt{-5}\rangle$ does not equal $\Z[\sqrt{-5}]$. Moreover, non-coprime reduced fractions appear to lack a unique representation, as may be seen by the example
\[\frac{1 + \sqrt{-5}}{2} = \frac{3}{1 - \sqrt{-5}}.\]
This ambiguity arises from $\Z[\sqrt{-5}]$ not being a UFD and $6 = 2 \cdot 3 = (1 + \sqrt{-5})(1 - \sqrt{-5})$.

\subsection{Extending the notion of density} Note that $\mathbb{Q}$ may be identified with
$$
\{(m,n)\in \mathbb{Z} \times \mathbb{N} : \gcd(m,n)=1\}.
$$
To study subsets of the rationals that are invariant under multiplication by a unit (i.e. $\pm 1$) and translation by an integer, it suffices to study subsets of
\begin{equation}\label{eq:DefIQ}
\{(m,n)\in \nicefrac{\Z}{n\Z} \times \mathbb{N} : \gcd(m,n)=1 \} =: \mathcal{I}.
\end{equation}
Classical definitions of natural and Dirichlet density for subsets of $\mathbb{N}$ may then be extended to subsets of $\mathcal{I}$ in a straightforward manner; see Section 2.1 for details.

Turning now to number fields, note that $K$ may be identified with
\[
\left\{(a,b)\in \mathcal{O}_K \times \nicefrac{\mathcal{O}_K \setminus \{0\}}{\mathcal{O}_K^\times} : \frac{a}{b} \textup{ is a reduced fraction.}\right\}.
\]
Similarly to the case of the rational numbers, we may define a density on subsets of
\begin{equation}\label{eq:DefI_K}
\left\{(a,b)\in \nicefrac{\mathcal{O}_K}{b\mathcal{O}_K} \times \nicefrac{\mathcal{O}_K \setminus \{0\}}{\mathcal{O}_K^\times} : \frac{a}{b} \textup{ is a reduced fraction.}\right\} =: \mathcal{I}_K.    
\end{equation}

Note that $\mathcal{I}_\Q = \mathcal{I}$.  We are then interested in the set of coprime reduced fractions, i.e.
\begin{equation}\label{eq:DefC_K}
    \mathscr{C}_K := \left\{(a,b)\in \mathcal{I}_K: \langle a, b \rangle = \mathcal{O}_K\right\} \subset \mathcal{I}_K.
\end{equation}
After extending the aforementioned notion of density to this setting, we will see that the density of $\mathscr{C}_K$ depends on a representative system of the class group of $K$.

\subsection{Main result.} In order to convey the property of being reduced to the ideals of a number field, we make the following definition.

\begin{definition}
    Let $K$ be a number field and $\mathfrak{a} \subset \mathcal{O}_K$ be an ideal. If the only principal ideal dividing $\mathfrak{a}$ is $\mathcal{O}_K$, we call $\mathfrak{a}$ \textit{inseverable}.
\end{definition}

With this in hand, we determine a way to uniquely choose a denominator for any reduced fraction in a number field.  Ordering elements of $\nicefrac{\mathcal{O}_K \setminus \{0\}}{\mathcal{O}_K^\times}$ by norm, we count a fraction as soon as it ``appears'' for the first time and no later.  The natural density of coprime fractions in the set of reduced fractions may then be computed using well-known facts about Hecke $L$-functions.

\begin{theorem} \label{theo:Densityestimate}
    The natural density of coprime fractions in a number field $K$ with respect to a set of inseverable representatives $\mathcal{R}_K$ of its class group equals $\frac{1}{\sum_{\mathfrak{g} \in \mathcal{R}_K} \frac{1}{(N\mathfrak{g})^2}}$.
\end{theorem}

Density theorems may be of interest in the study of Diophantine approximation in number fields, where Palmer has recently proved number field analogues of classical results in metric number theory \cite{Palmer1,Palmer2}. For a general reference, see \cite{Harman}.

\section*{Acknowledgements}
The first author received funding from the European Research Council (ERC) under the European Union's Horizon 2020 research and innovation programme (grant agreement No. 101001179), and the second author is partially supported by the Alfried Krupp prize.

\section*{Notation} The letters $\N$, $\Z$, $\Q$, $\R$, and $\C$ have their usual meaning. We denote by
\[\HH := \{\tau = u + iv \in \C: u, v \in \R, v > 0\}\]
the complex upper half-plane. Furtheremore, the ring of integers of a number field $K$ is denoted by $\mathcal{O}_K$, its class group by $\operatorname{Cl}_K$, and the {\it absolute value} of the norm of an algebraic integer $\alpha$ by $N \alpha$. $[\mathfrak{a}]$ represents the class of a fractional ideal $\mathfrak{a}$ of $K$ in $\operatorname{Cl}_K$. We also abbreviate
\begin{align*}
\langle\omega_1, \omega_2, \ldots, \omega_n\rangle_R := R \omega_1 + R\omega_2 + \cdots + R\omega_n
\end{align*}
for any ring $R$ and omit the index if $R = \mathcal{O}_K$.

\section{Density theory}

\subsection{Density in ordered countable sets} 
Given $\mathcal{A} = \{a_1, a_2, \ldots \} \supset \{b_1, b_2, \ldots \} = \mathcal{B}$, one can say that the natural density of $\mathcal{B}$ in $\mathcal{A}$ is $\delta$ if
\begin{align} \label{eq:ordered-density}
\lim_{n \to \infty} \frac{|\mathcal{B} \cap \{a_1, \dots, a_n\}|}{n} = \delta.
\end{align}
More generally, within $\mathcal{A} $, we assign an attribute $n$ to each element via a mapping $\kappa \colon \mathcal{A} \to \N$, such as assigning to a coprime positive fraction the value of the denominator. The number of elements with attribute $n$ then defines a sequence $f(n)$. If this sequence has moderate growth, the corresponding Dirichlet series 
\begin{align*}
F(s) := \sum_{n \geq 1} \frac{f(n)}{n^s}
\end{align*}
converges on a right half-plane. Since $\mathcal{A}$ is infinite, by Landau's theorem there exists a singular point $\sigma_{\mathcal{A}} \geq 0$, such that the function $F$ cannot be holomorphically continued to the left of the half plane $\mathrm{Re}(s)>\sigma_{\mathcal{A}}$. In some applications this is a (multiple) pole, so at least a meromorphic continuation is possible, which has some advantages.
\begin{proposition}[see \cite{Tenenbaum} on p. 350]\label{P:Delange}
Let $F(s):=\sum_{n\geq 1}\frac{f(n)}{n^s}$ be a Dirichlet series with nonnegative coefficients, converging for $\mathrm{R}(s)>\sigma_0>0$, such that the function
\begin{align*}
			H(s) := F(s) - \frac{A}{s - \sigma_0}
\end{align*}
with some real $A$ has an analytic continuation to $\{s \in \C : \mathrm{Re}(s) \geq \sigma_0 \}$. Then we have 
\begin{align*}
\sum_{1 \leq n \leq x} f(n) \sim \frac{Ax^{\sigma_0}}{\sigma_0}, \qquad x \rightarrow \infty.
\end{align*}
\end{proposition}

We now define two notions of density.

\begin{definition}  Let $\mathcal{B} \subseteq \mathcal{A}$.  Define \begin{align*}
g(n) := |\{ x \in \mathcal{B} \colon \kappa(x) = n\}|.
\end{align*} Note that $0 \leq g(n) \leq f(n)$.  When the limit exists, we define a {\it natural} density,
    \begin{align*}
\delta_{\mathrm{nat}}(\mathcal{B}):=\lim_{x \to \infty} \frac{\sum_{n \leq x} g(n)}{\sum_{n \leq x} f(n)} = \delta
\end{align*}
and a {\it Dirichlet} density
\begin{align*}
\delta_{\mathrm{Dir}}(\mathcal{B}):=\lim_{s \to \sigma_{\mathcal{A}}^+} \frac{\sum_{n \geq 1} \frac{g(n)}{n^s}}{\sum_{n \geq 1} \frac{f(n)}{n^s}} = \delta.
\end{align*}
\end{definition}

Note that \eqref{eq:ordered-density} is in some ways a more precise definition than the upper one, but requires a total ordering $<_{\mathcal{A}}$ on the set $\mathcal{A}$ that respects the attribute map $\kappa \colon \mathcal{A} \to \N$, i.e., $a_m <_{\mathcal{A}} a_n$ if and only if $\kappa(a_m) < \kappa(a_n)$. However, since in the above approach we do not order within element classes with the same attribute, this is a significant refinement with respect to the upper requirements. However, under mild conditions at $\mathcal{A}$ and $\kappa$, the two definitions can be merged, i.e., the sorting of the elements within the original images $\kappa^{-1}(\{n\})_{n \in \N}$ is therefore no longer important. Such a condition is 
$$f(n) = o\left( \sum_{1 \leq j \leq n-1} f(j)\right),$$
in other words, that there is a certain ``uniformity'' in the distribution.

As in the classical case, one can show that when the natural density exists, so does the Dirichlet density and both are equal. The proof is a straightforward calculation using partial summation. 

\begin{proposition} \label{prop:NatDensImpliesDirichGeneral}
Let the notation be as above. Assume that the Dirichlet series $\sum_{n \geq 1} f(n)n^{-s}$ has abscissa of convergence $\sigma_{\mathcal{A}} > 0$ and that $\sum_{n \geq 1} f(n)n^{-\sigma_{\mathcal{A}}} = \infty$. Further assume that $\mathcal{B} \subseteq \mathcal{A}$ has a natural density, i.e.,
\begin{align*}
\lim_{x \to \infty} \frac{\sum_{n \leq x} g(n)}{\sum_{n \leq x} f(n)} = \delta
\end{align*}
for some $\delta \in [0,1]$. Then 
\begin{align*}
\lim_{s \to \sigma_{\mathcal{A}}^+} \frac{\sum_{n \geq 1} \frac{g(n)}{n^s}}{\sum_{n \geq 1} \frac{f(n)}{n^s}} = \delta.
\end{align*}
\end{proposition}

\subsection{Density in $\Q$} Recall the set $\mathcal{I}$ from the introduction (see \ref{eq:DefIQ}).  Here, we take $\kappa(m,n):=n$ and $f(n)=\varphi(n)$.  Let $\mathcal{B} \subset \mathcal{I}$. We have a {\it natural} density
    \begin{align} \label{eq:naturaldensity}
    \delta_{\mathrm{nat}}(\mathcal{B}) = \lim_{x \to \infty} \frac{\sum_{1 \leq n \leq x}\sum_{1 \leq m \leq n}\mathbf{1}_\mathcal{B}(m,n)}{\sum_{1\leq n \leq x} \varphi(n)} = \lim_{x \to \infty} \frac{\pi^2}{3x^2}\sum_{1 \leq n \leq x}\sum_{1 \leq m \leq n}\mathbf{1}_\mathcal{B}(m,n),
    \end{align}
    and a {\it Dirichlet} density
    \begin{align} 
    \nonumber \delta_{\mathrm{Dir}}(\mathcal{B}) &:= \lim_{s \to 2^+} \frac{\zeta(s)}{\zeta(s-1)} \sum_{n \geq 1} \frac{1}{n^s} \sum_{1 \leq m \leq n} \mathbf{1}_\mathcal{B}(m,n) \\ 
    \label{eq:dirichletdensity} &=\lim_{s \to 2^+} \frac{\pi^2}{6\zeta(s-1)} \sum_{n \geq 1} \frac{1}{n^s} \sum_{1 \leq m \leq n} \mathbf{1}_\mathcal{B}(m,n).
    \end{align}

\begin{example}
    Let $\mathcal{B}_p := \{(a,n) \in \mathcal{I}: p \mid n\}$ for a rational prime $p \in \mathbb{P}$.  Here,
    $$
    g(n)=\mathbf{1}_{p \mid n} \varphi(n).
    $$Then
    \begin{align*}
        \delta_\mathrm{nat}(\mathcal{B}_p) &= \lim_{x \to \infty} \frac{\pi^2}{3x^2}\sum_{1 \leq n \leq x, \, p \mid n} \varphi(n)\\
        &= \lim_{x \to \infty} \frac{\pi^2}{3x^2} \left(\sum_{1 \leq n \leq \frac{x}{p}, \, p \mid n}p\varphi(n) + \sum_{1 \leq n \leq \frac{x}{p}, \, p \nmid n}(p - 1)\varphi(n)\right)\\
        &= \frac{1}{p^2} \lim_{x \to \infty} \frac{\pi^2}{3\left(\frac{x}{p}\right)^2}\left(\sum_{1 \leq n \leq \frac{x}{p}, \, p \mid n}\varphi(n) + (p - 1)\sum_{1 \leq n \leq \frac{x}{p}} \varphi(n)\right)\\
        &= \frac{1}{p^2}(\delta_\mathrm{nat}(\mathcal{B}_p) + p - 1).
    \end{align*}
    Thus, $\delta_\mathrm{nat}(\mathcal{B}_p) = \frac{1}{p + 1}$.
\end{example}

\subsection{Density in a number field}

    Let $K$ be a number field and recall the set $\mathcal{I}_K$ from the introduction (see \eqref{eq:DefI_K}).  Below, sums over $b$ run as in \eqref{eq:DefI_K}, i.e. over $\nicefrac{\mathcal{O}_K \setminus \{0\}}{\mathcal{O}_K^\times}$.   Denote the number of reduced fractions in $\mathcal{I}_K$ with denominator $b$ with respect to $\mathcal{R}_K$ by $\eta_K(b)$.  Here, we take $\kappa(a,b):=N b$ and $$f(n)=\widetilde{\eta_K}(n):= \sum_{N b = n}\eta_K(b).$$
    For a subset $\mathcal{S} \subset \mathcal{I}_K$, we set $\mathcal{S}(b) := \#\{y \in \mathcal{S}: y \textup{ has denominator } b\}$.  We have a {\it natural} density
        \[\delta_{\mathrm{nat},K}(\mathcal{S}) := \lim_{x \to \infty} \frac{\sum_{1\leq n \leq x} \sum_{ Nb=n}  \sum_{y \pmod{b}}\mathbf{1}_{\mathcal{S}}(y,b)}{\sum_{1\leq n \leq x} \widetilde{\eta_K}(n)}= \lim_{x \to \infty} \frac{\sum_{Nb \leq x}   \mathcal{S}(b)}{\sum_{1\leq n \leq x} \widetilde{\eta_K}(n)},\]
         a \textit{Dirichlet} density by
        \[\delta_{\mathrm{Dir},K}(\mathcal{S}) := \lim_{s \to \sigma_\eta} \frac{\sum_{n \geq 1} \sum_{ Nb=n}  \sum_{y \pmod{b}}\frac{\mathbf{1}_{\mathcal{S}}(y,b)}{n^s}}{\sum_{n \geq 1} \frac{\widetilde{\eta_K}(n)}{n^s}}= \lim_{s \to \sigma_\eta} \frac{\sum_{b} \frac{\mathcal{S}(b)}{(Nb)^s}}{\sum_{b} \frac{\eta_K(b)}{(Nb)^s}}\]
        where $\sigma_\eta$ is the abscissa of convergence of $\sum_{b} \eta_K(b)(Nb)^{-s}$.

        In particular, when $\mathcal{S}=\mathscr{C}_K\subset \mathcal{I}_K$ is the set of coprime fractions, we have $\mathcal{S}(b)=\varphi(\langle b \rangle)$, and
        \begin{equation}\label{E:DirDensityCorpime}
        \delta_{\mathrm{Dir},K}(\mathscr{C}_K)= \lim_{s \to \sigma_\eta} \frac{\sum_{b} \frac{\varphi(\langle b \rangle)}{(Nb)^s}}{\sum_{b} \frac{\eta_K(b)}{(Nb)^s}}.
        \end{equation}

\section{Ideal theory and reduced fractions}

Recall the sets $\mathcal{I}_K$ and $\mathscr{C}_K$ from the introduction (see \eqref{eq:DefI_K} and \eqref{eq:DefC_K}). In order to let the density on $K$ imitate what we intuitively expected from $\delta_\mathrm{nat}$ and $\delta_\mathrm{Dir}$, we require the elements of $\mathcal{R}_K$ to be inseverable. In particular, the only inseverable principal ideal is $\mathcal{O}_K$ itself.
\begin{lemma}\label{lem:RedIfInsev}
    The fraction $\frac{a}{b}$ with $b \neq 0$ is reduced if and only if $\langle a, b \rangle$ is an inseverable ideal.
\end{lemma}

\begin{proof}
    Suppose that $\langle a, b \rangle$ is not inseverable, i.e. there exists $c \in \mathcal{O}_K \setminus \mathcal{O}_K^\times$ such that $\langle a, b \rangle \subset \langle c \rangle \subsetneq \mathcal{O}_K$. Then $c \mid a$ and $c \mid b$ so that $\frac{a}{b}$ can be reduced to $\nicefrac{\frac{a}{c}}{\frac{b}{c}}$. Hence, $\frac{a}{b}$ is not reduced.
    
    We are left to show that $\langle a, b \rangle$ being inseverable is sufficient for $\frac{a}{b}$ being reduced. So conversely, suppose that $\langle a, b \rangle$ is an inseverable ideal of $\mathcal{O}_K$. Then this immediately implies that there exists no $c \in \mathcal{O}_K$ with $c \mid a$ and $c \mid b$ because, otherwise, $\langle c \rangle \supset \langle a, b \rangle$ would hold.
\end{proof}

A useful sufficient criterion to the set $\mathcal{R}_K$ to contain exclusively inseverable ideals is given by the next lemma.

\begin{proposition} \label{prop:inseverable-ideals}
    Let $\mathcal{R}_K$ be a system of representatives of the class group such that each representative $\mathfrak{g} \in \mathcal{R}_K$ is an integral ideal with minimal norm within its respective class. Then $\mathcal{R}_K$ is a set of inseverable ideals. 
\end{proposition}

\begin{proof}
    By means of contradiction, suppose $\mathfrak{g} \in \mathcal{R}_K$ is not inseverable. Then there is an $a \in \mathcal{O}_K$ such that $\mathcal{O}_K \supsetneq \langle a \rangle \mid \mathfrak{g}$ and $\frac{\mathfrak{g}}{\langle a \rangle}$ is an integral ideal lying in the same ideal class as $\mathfrak{g}$ but with lesser norm. 
\end{proof}

We will now collect the reduced fractions of a number field in equivalence classes and hence determine how to ensure that each fraction is counted at most once.

\begin{proposition}\label{prop:ClassesOfFrations}
    Suppose $\frac{a_1}{b_1}$ is a reduced fraction where $b_1 \not \in \mathcal{O}_K^\times \cup \{0\}$. Suppose further that $\langle a_1, b_1\rangle = \mathfrak{g}_1$ is a non-principal inseverable ideal and $b_2 \neq b_1$ is another given element of $\mathcal{O}_K \setminus \{0\}$. Then $\frac{a_1}{b_1} = \frac{a_2}{b_2}$ (for some $a_2 \in \mathcal{O}_K$) is another reduced fraction if and only if there is an inseverable ideal $\mathfrak{g}_2 \neq \mathfrak{g}_1$ such that
    \[\frac{\langle b_1 \rangle}{\mathfrak{g}_1} = \frac{\langle b_2 \rangle}{\mathfrak{g}_2}.\]
    In particular, $[\mathfrak{g}_1] = [\mathfrak{g}_2]$.
\end{proposition}

\begin{proof}
    If $\frac{a_1}{b_1} = \frac{a_2}{b_2}$, then $a_2 = \frac{a_1b_2}{b_1}$ and by Lemma \ref{lem:RedIfInsev}, there exists an inseverable ideal $\mathfrak{g}_2 \mid b_2$ such that
    \[\mathfrak{g}_2 = \langle a_2, b_2 \rangle = \left\langle \frac{a_1b_2}{b_1}, b_2 \right\rangle = \langle b_1 \rangle^{-1} \langle b_2 \rangle \langle a_1, b_1 \rangle = \langle b_1 \rangle^{-1} \langle b_2 \rangle \mathfrak{g}_1.\]
    Conversely, given $\frac{\langle b_1\rangle }{\mathfrak{g}_1}=\frac{\langle b_2\rangle}{\mathfrak{g}_2}$, we again get 
    $$
    \mathfrak{g}_2=\langle a_2, b_2\rangle=\left< \frac{a_1b_2}{b_1}, b_2\right>,
    $$
and by taking $a_2=\frac{a_1b_2}{b_1}$ we get another fraction which is reduced because $\mathfrak{g}_2$ is inseverable.
\end{proof}

\begin{remark}
    If $\frac{a}{b}$ is a reduced fraction, then $\langle ca, cb \rangle = \langle c \rangle \langle a, b \rangle$. Hence, the last sentence of the previous proposition remains true for non-reduced fractions.  Moreover, if $\langle a_1, b_1 \rangle$ is a principal inseverable ideal, then it equals $\mathcal{O}_K$ and the statement becomes trivial.
\end{remark}

Now we conversely prove that each fraction is counted at least once.

\begin{proposition}\label{prop:EveryIdealIsAFraction}
    Let $\mathfrak{a} \subset \mathcal{O}_K$ be a non-zero ideal. Then there is a fraction $\frac{a}{b}$ with $a, b \in \mathcal{O}_K$ such that $\langle a, b \rangle = \mathfrak{a}$.
\end{proposition}

\begin{proof}
    Take any $a \in \mathfrak{a} \setminus \{0\}$, so $\mathfrak{a} \supset \langle a \rangle$. As $\mathcal{O}_K$ is a Dedekind ring, the ideal $\langle a \rangle$ has a unique ideal factorization $\langle a \rangle = \mathfrak{a} \prod_{k = 1}^r\mathfrak{p}^{e_k}$. Now take any ideal $\mathfrak{b} \in \left[\prod_{k = 1}^r\mathfrak{p}^{e_k}\right]$ being coprime to $\prod_{k = 1}^r\mathfrak{p}^{e_k}$, so $\mathfrak{a}\mathfrak{b}$ is a principal ideal whose generator we denote by $b$. Then $\langle a, b \rangle = \mathfrak{a}$.
\end{proof}


Finally, we will see that our density depends on such a system of representatives.

\begin{lemma}\label{lem:NumberOfRedFracModB}
    Let $\mathcal{R}_K$ be a system of representatives of the class group such that each of its elements is an inseverable ideal. Suppose further that $b \in \mathcal{O}_K \setminus \{0\}$. Then the number of reduced fractions in $\mathcal{I}_K$ with denominator $b$ with respect to $\mathcal{R}_K$ equals
    \[ \eta_K(b) := \sum_{\mathfrak{g} \in \mathcal{R}_K} \varphi\left(\frac{\langle b \rangle}{\mathfrak{g}}\right)\]
    where $\varphi$ is Euler's totient function for number fields which vanishes on non-integral ideals.
\end{lemma}

\begin{proof}
    By Lemma \ref{lem:RedIfInsev} as well as Propositions \ref{prop:ClassesOfFrations} and \ref{prop:EveryIdealIsAFraction}, we want to compute
   $$\# \{a \pmod{b} : \langle a,b \rangle = \mathfrak{g}\}.$$
     Let $\pi\in \mathfrak{p}^{e} \setminus \mathfrak{p}^{e+1}$ for primes $\mathfrak{p}$ with $\mathfrak{p}^e \mid \mid \mathfrak{g}$ and otherwise $\pi \notin \mathfrak{p}$ for primes with $\mathfrak{p} \mid b$ and $\mathfrak{p} \nmid \mathfrak{g}$.  Consider the map
    \begin{align*}
        \psi:\left\{a \pmod{\frac{b}{\mathfrak{g}}} : \left\langle a,\frac{b}{\mathfrak{g}} \right\rangle = \mathcal{O}_K\right\} \to \{a \pmod{b} : \langle a,b \rangle = \mathfrak{g}\},\\
        \text{defined by $f(a):=\pi a \pmod{b}$.}
    \end{align*}
    We claim $\psi$ is a bijection.  To check that $\psi$ is well-defined, let $\gamma \in \frac{b}{\mathfrak{g}}$.  Then $\psi(a+\gamma)=\pi a + \pi \gamma \equiv \pi a \pmod{b}=\psi(a)$.  For injectivity, suppose $\langle a_1,\frac{b}{\mathfrak{g}} \rangle=\langle a_2,\frac{b}{\mathfrak{g}} \rangle=1$.  Now if $\pi a_1 \equiv \pi a_2 \pmod{b}$, then $b \mid \pi(a_1-a_2)$ and by the choice of $\pi$, one has $\frac{b}{\mathfrak{g}} \mid (a_1-a_2)$, so $a_1 \equiv a_2 \pmod{\frac{b}{\mathfrak{g}}}$.  Finally, for surjectivity, let $a_2$ satisfy $\langle a_2,b \rangle = 1$.  Then by the definition of $\pi$, we see that the fractional ideal $\langle \frac{a_2}{\pi} \rangle$ contains no powers of primes dividing $b$ in its prime factorization.  Hence, there exists $a_1 \equiv \frac{a_2}{\mathfrak{g}} \pmod{\frac{b}{\mathfrak{g}}}$ with $\langle a_1, \frac{b}{\mathfrak{g}} \rangle =1,$ which clearly gives $\psi(a_1) \equiv a_2 \pmod{b}$.
\end{proof}

\begin{example}
    Let $K = \Q(\sqrt{-5})$, so $h_K = 2$ and we choose $\mathcal{R}_K = \{\mathcal{O}_K,\mathfrak{p}_2\}$ where $\mathfrak{p}_2 = \langle 2, 1 + \sqrt{-5}\rangle$. Choose $b = 2$. Invoking Lemma \ref{lem:NumberOfRedFracModB} we get that
    \[\eta_K(2) = \varphi(\langle 2 \rangle) + \varphi\left(\frac{\langle 2 \rangle}{\mathfrak{p}_2}\right) = \varphi(\mathfrak{p}_2^2) + \varphi(\mathfrak{p}_2) = 2 + 1 = 3.\]
    Noting that $N(2) = 4$, we quickly see that a representative system of residues modulo $\langle 2 \rangle$ is given by $\{0,1,\sqrt{-5},1 + \sqrt{-5}\}$. We check that among the corresponding ideals $\mathfrak{p}_2^2 = \langle 2 \rangle, \mathcal{O}_K, \mathcal{O}_K,$ and $\mathfrak{p}_2$ all but the first one are inseverable and hence are tied to reduced fractions. Furthermore, we see that the coprime residues classes are exactly those tied to $\mathcal{O}_K$, so $1$ and $\sqrt{-5}$, and the non-coprime reduced fraction is $\frac{1 + \sqrt{-5}}{2}$.
\end{example}

\section{Proof of Theorem \ref{theo:Densityestimate}}

In order to proceed with our considerations regarding density, we need explicit expressions of the associated Dirichlet series. Using Euler product expansions, the following lemma is straightforward. 

\begin{lemma} \label{lem:L-EulerPhi} Let $\chi \colon \mathrm{Cl}_K \to \C^\times$ be a class group character. Then we have 
\begin{align*}
\sum_{\mathfrak{a} \subseteq \mathcal{O}_K} \frac{\chi([\mathfrak{a}]) \varphi(\mathfrak{a})}{(N \mathfrak{a})^s} = \frac{L(s - 1; \chi)}{L(s; \chi)},
\end{align*}
where $\varphi$ is Euler's totient for number fields and 
\begin{align*}
L(s; \chi) := \sum_{\mathfrak{a} \subseteq \mathcal{O}_K} \frac{\chi([\mathfrak{a}])}{(N \mathfrak{a})^s}
\end{align*}
is the $L$-function corresponding to the character $\chi$. 
\end{lemma}
It is well-known that the $L$-functions $L(s;\chi)$ in Lemma \ref{lem:L-EulerPhi} have a holomorphic continuation to $\C \setminus \{1\}$ and a simple pole in $s=1$ if and only if the character $\chi$ is trivial. Indeed, they correspond to Hecke $L$-functions of the Hilbert class field $H$ of $K$ due to the isomorphism $\mathrm{Gal}(H/K) \cong \mathrm{Cl}_K$. The Euler product expansion further implies that $L(s;\chi) \not= 0$ if $\mathrm{Re}(s) > 1$. We can use these Dirichlet series to find the Dirichlet density of coprime fractions. Note first that we have by Lemma \ref{lem:L-EulerPhi} and Lemma \ref{lem:NumberOfRedFracModB}
\begin{align}
\nonumber \mathrm{H}_K(s) := \sum_{n \geq 1} \frac{\Tilde{\eta_K}(n)}{n^s} := \sum_{b \in \mathcal{O}_K/\mathcal{O}_K^{\times}} \frac{1}{Nb^s}\left(\sum_{\mathfrak{g} \in \mathcal{R}_K}\varphi\left(\frac{\langle b\rangle}{\mathfrak{g}}\right) \right) &= \sum_{\mathfrak{g} \in \mathcal{R}_K} \frac{1}{(N\mathfrak{g})^s} \sum_{\mathfrak{b} \in [\mathfrak{g}]^{-1}} \frac{\varphi(\mathfrak{b})}{(N\mathfrak{b})^s} \\
\label{eq:Dirichlet-series} &=\frac{1}{h_K}\sum_{\substack{\mathfrak{g} \in \mathcal{R}_K \\ \chi \in \widehat{\mathrm{Cl}_K}}} \frac{\chi([\mathfrak{g}])}{(N\mathfrak{g})^s} \frac{L(s-1; \chi)}{L(s; \chi)}.
\end{align}

For the counting method to define the density, it is natural to choose only representatives for $\mathcal{R}_K$ so that they have minimal norm in their class. That is because we count a fraction as soon as it ``appears'' for the first time and no later. In particular, it is always assumed that $\mathcal{O}_K \in \mathcal{R}_K$ in the following, since by Proposition \ref{prop:inseverable-ideals} such a system of representatives consists of inseverable ideals and $\mathcal{O}_K$ is the only inseverable principal ideal.  


\begin{proof}[Proof of Theorem \ref{theo:Densityestimate}] 

Recall \eqref{E:DirDensityCorpime}, and consider the Dirichlet series  
\begin{align*}
\Phi(s) := \sum_{n \geq 1} \frac{\Tilde{\varphi}(n)}{n^s} := \sum_{b \in \nicefrac{\mathcal{O}_K \setminus \{0\}}{\mathcal{O}_K^\times} } \frac{\varphi(\langle b \rangle)}{Nb^s} = \frac{1}{h_K}\sum_{\substack{\chi \in \widehat{\mathrm{Cl}_K}}} \frac{L(s-1; \chi)}{L(s; \chi)}.
\end{align*}
By Proposition \ref{P:Delange}, we therefore have
\begin{align*}
\sum_{n \leq x} \Tilde{\varphi}(n) \sim \frac{\mathrm{Res}_{s=2} \Phi(s) x^2}{2}, \qquad x \to \infty.
\end{align*}
On the other hand, a similar argument yields with \eqref{eq:Dirichlet-series}
\begin{align*}
\sum_{n \leq x} \Tilde{\eta_K}(n) \sim \frac{\mathrm{Res}_{s=2} \mathrm{H}_K(s) x^2}{2}.
\end{align*}
Now note that 
\begin{align*}
\mathrm{Res}_{s = 2} \Phi(s) = \frac{\mathrm{Res}_{s=2} L(s-1; \chi_0)}{h_K L(2; \chi_0)}, \quad \mathrm{Res}_{s = 2} \mathrm{H}_K(s) = \sum_{\mathfrak{g} \in \mathcal{R}_K} \frac{1}{(N\mathfrak{g})^2} \frac{\mathrm{Res}_{s=2} L(s-1; \chi_0)}{h_K L(2; \chi_0)}.
\end{align*}
Now the claim follows. 
\end{proof}
By Proposition \ref{prop:NatDensImpliesDirichGeneral} the existence of the natural density implies a Dirichlet density with the same value. For this reason we can simply speak about density without further specification. 
\begin{corollary}\label{coro:DensityBounds}
     The density in Theorem \ref{theo:Densityestimate} is at least $\frac{4}{3 + h_K}$ where $h_K$ is the class number of $K$. If there exists a $B \in \N$ such that there is no non-principal ideal $\mathfrak{a} \subset \mathcal{O}_K$ with $N(\mathfrak{a}) \leq B$, this bound improves to $\frac{(B + 1)^2}{B^2 + 2B + h_K}$. Moreover, if the elements of $\mathcal{R}_K$ are of minimal norm within their ideal class, an upper bound for the density is given by
     \[\frac{1}{1 + \frac{h_K - 1}{|d_K|}\left(\frac{n^n\pi^{r_2}}{n!4^{r_2}}\right)^2}\]
     where $d_K$ is the discriminant of $K$ over $\Q$, $n = [K:\Q]$, and $r_2$ is the number of pairs of complex embeddings of $K$.
\end{corollary}

\begin{proof}
    Note that the existence of such a $B$ implies $N \mathfrak{g} \geq B + 1$ for all $\mathfrak{g} \in \mathcal{R}_K \setminus \{\mathcal{O}_K\}$. By Theorem \ref{theo:Densityestimate}, we have
    \[\delta_{\mathrm{nat},K}(\mathscr{C}_K) = \frac{1}{\sum_{\mathfrak{g} \in \mathcal{R}_K} \frac{1}{(N\mathfrak{g})^2}} \geq \frac{1}{1 + (h_K - 1) \frac{1}{(B + 1)^2}} = \frac{(B + 1)^2}{B^2 + 2B + h_K},\]
    as claimed. Replacing $B$ by 1 yields the unconditional lower bound.

    For the upper bound, we simply use Minkowski's bound which can be found in \cite{Jarvis} on p. 161. That is, every ideal class contains an integral ideal of norm not exceeding
    \[\sqrt{|d_K|}\left(\frac{4}{\pi}\right)^{r_2}\frac{n!}{n^n},\]
    yielding the claim.
\end{proof}
\begin{remark} \
    \begin{enumerate}
        \item The example $K = \Q(\sqrt{-5})$ shows that this lower bound is sharp even if $K$ does not have class number 1. As $h_{\Q(\sqrt{-5})} = 2$ and the ideal $\mathfrak{p}_2 := \langle 2, 1 + \sqrt{-5}\rangle$ is non-principal with $N(\mathfrak{p}_2) = 2$, we choose $\mathcal{R}_K = \{\mathcal{O}_K, \mathfrak{p}_2\}$. This yields
        \[\delta_{\mathrm{nat},\Q(\sqrt{-5})}(\mathscr{C}_{\Q(\sqrt{-5})}) = \frac{1}{\frac{1}{1^2} + \frac{1}{2^2}} = \frac{4}{5}.\]
        We evaluate the coprimality for some $b$ of small norm in Table \ref{tab:NumEvidDensityQsqrt(-5)}.
        \begin{table}[h]
            \centering
            \[
            \begin{array}{c|c|c|c|c|c|c|c}
                b & N(b) & \varphi(\langle b \rangle) &\varphi\left(\frac{\langle b \rangle}{\mathfrak{p}_2}\right) & \eta_K(b) & \frac{\varphi(\langle b \rangle)}{\eta_K(b)} &  \frac{\sum_{N(c) \leq N(b)} \varphi(\langle c \rangle)}{\sum_{N(c) \leq N(b)} \eta_K(c)} & \approx\\
                \hline
                1 & 1 & 1 & 0 & 1 & 1 & 1 & 1\\
                2 & 4 & 2 & 1 & 3 & \nicefrac{2}{3} & \nicefrac{3}{4} & .75\\
                \sqrt{-5} & 5 & 4 & 0 & 4 & 1 & \nicefrac{7}{8} & .875\\
                1 + \sqrt{-5} & 6 & 2 & 2 & 4 & \nicefrac{1}{2} & & \\
                1 - \sqrt{-5} & 6 & 2 & 2 & 4 & \nicefrac{1}{2} & \nicefrac{11}{16} & .6875\\
                3 & 9 & 4 & 0 & 4 & 1 & & \\
                2 + \sqrt{-5} & 9 & 6 & 0 & 6 & 1 & & \\
                2 - \sqrt{-5} & 9 & 6 & 0 & 6 & 1 & \nicefrac{27}{32} & .8437\\
                3 + \sqrt{-5} & 14 & 6 & 6 & 12 & \nicefrac{1}{2} & & \\
                3 - \sqrt{-5} & 14 & 6 & 6 & 12 & \nicefrac{1}{2} & \nicefrac{39}{56} & .6964\\
                4 & 16 & 8 & 4 & 12 & \nicefrac{2}{3} & \nicefrac{47}{68} & .6911 \\
                2\sqrt{-5} & 20 & 8 & 4 & 12 & \nicefrac{2}{3} & \nicefrac{55}{80} & .6875 \\
                1 + 2\sqrt{-5} & 21 & 12 & 0 & 12 & 1 & & \\
                1 - 2\sqrt{-5} & 21 & 12 & 0 & 12 & 1 & & \\
                4 + \sqrt{-5} & 21 & 12 & 0 & 12 & 1 & & \\
                4 - \sqrt{-5} & 21 & 12 & 0 & 12 & 1 & \nicefrac{103}{128} & .8046
            \end{array}
            \]
            \caption{Some numerical values for $K = \Q(\sqrt{-5})$}
            \label{tab:NumEvidDensityQsqrt(-5)}
        \end{table}
        See Figure \ref{fig:1} for a plot of the first 10\hspace{.7mm}000 values.
        \begin{figure}[H]
            \centering
            \includegraphics[width=120mm,height= 55mm]{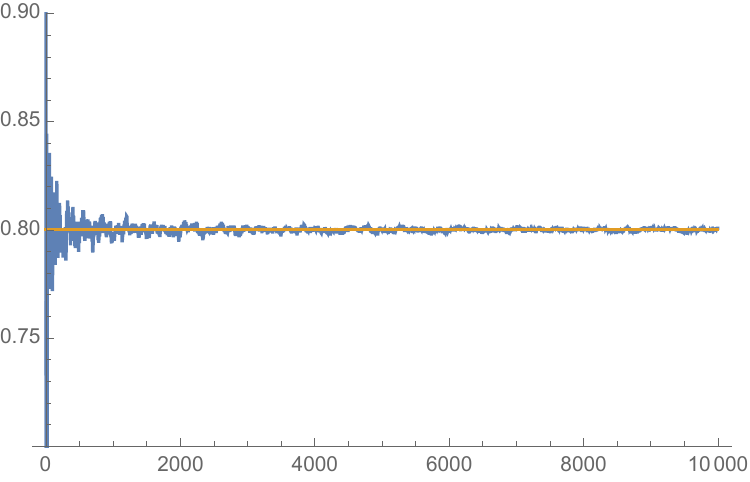}
            \caption{The first 10\hspace{.7mm}000 values of $\frac{\sum_{N(c) \leq N(b)} \varphi(\langle c \rangle)}{\sum_{N(c) \leq N(b)} \eta_K(c)}$}
            \label{fig:1}
        \end{figure}
        Note that in the case of $h_K = 2$, the function $\eta_K$ is multiplicative which notably simplified computations.
        \item Depending on the degree of the extension $K/\Q$, the lower bound of Corollary \ref{coro:DensityBounds} can easily be improved. That is, if $[K:\Q] = n$, then there are at most $n$ prime ideals of a given norm which also limits the number of ideals of any given norm. E.g., in the case of a quadratic number field of class number 5, we have
        \[\delta_{\mathrm{nat},K}(\mathscr{C}_K) \geq \frac{1}{\frac{1}{1^2} + 2 \cdot \frac{1}{2^2} + 2 \cdot \frac{1}{3^2}} = \frac{18}{31}\]
        instead of the bound $\frac{1}{2}$, as provided by the corollary.
        \item Let $d \equiv 5 \pmod 8$, so 2 is inert in $\Q(\sqrt{d})/\Q$. In this case, the lower bound immediately improves to $\frac{9}{8 + h_{\Q(\sqrt{d})}}$. Using that an odd rational is inert in $\Q(\sqrt{d})$ if and only if $\leg{d}{p} = -1$ we can apply the Chinese Remainder Theorem to find a sequence $(d_n)_{n \in \N}$ such that the density of coprime fractions in $\Q(\sqrt{d_n})$ approaches 1 as $n \to \infty$ if we can bound $h_{\Q(\sqrt{d_n})}$ suitably. Whether this last condition is achievable has remained unknown to the authors.
    \end{enumerate}
\end{remark}

\section{Further results}

\subsection{Some more ideal theory}
    For a number field $K$, let $\mathcal{I}_K$ be the set of inseverable ideals of $K$. If $K$ has class number 2, $\mathfrak{I}_K$ has a plain description. It holds
    \[\mathfrak{I}_K = \left\{\mathfrak{p} \in \operatorname{Spec}(\mathcal{O}_K) \setminus \{\langle 0 \rangle\}: [\mathfrak{p}] \neq [1]\right\} \cup \{\mathcal{O}_K\}\]
    where $\operatorname{Spec}(\mathcal{O}_K)$ is the spectrum of $\mathcal{O}_K$. In the more general case, we can derive the set of inseverable ideals by applying a greedy algorithm starting with the set above for arbitrary $K$ and taking products that are still inseverable.

    Furthermore, we have the following key proposition.

\begin{proposition}
    \label{prop:InvarianceModTransf}
    Let $K$ be any number field and $\tau = \frac{a}{b}$ with $a, b \in \mathcal{O}_K, b \neq 0$. Let further $M \in \operatorname{SL}_2(\Z)$. Then $\langle a', b' \rangle = \langle a, b \rangle$ where $\frac{a'}{b'} = M\tau$.
\end{proposition}

\begin{proof}
    We only have to check the proposition for the well-known generators 
$$ T := \begin{pmatrix} 1 & 1 \\ 0 & 1\end{pmatrix}, \quad S := \begin{pmatrix} 0 & -1 \\ 1 & 0\end{pmatrix}$$
of the full modular group. It is $T\tau = \frac{a + b}{b}$ and $\langle a + b, b \rangle = \langle a, b \rangle$ as well as $S\tau = - \frac{b}{a}$ and $\langle -b, a \rangle = \langle a, b \rangle$. 
\end{proof}

\begin{remark}
    Replacing $\operatorname{SL}_2(\Z)$ in Proposition \ref{prop:InvarianceModTransf} by $\operatorname{SL}_2(\mathcal{O}_K)$ leads to the stronger statement
    \[\mathrm{Cl}(K) \cong \nicefrac{\mathbb{P}^1(K)}{\operatorname{SL}_2(\mathcal{O}_K)},\]
    cf. \cite[Lemma 1.3]{bru123} for a proof in the case of a real quadratic field.
\end{remark}

\subsection{Quadratic fields}

Recall the following definition which will be useful later on in this chapter.

\begin{definition}
    Let $K$ be a number field. An \textit{order} $\mathcal{O}$ is a subring of $K$ being finitely generated over $\Z$ and containing a $\Q$-basis of $K$.
\end{definition}

Because of numerous additional structures, the case of quadratic extensions is of special interest. We briefly recall some basic facts. Let $K = \Q(\sqrt{N})$ be a quadratic field with $N \in \Z \setminus \{0,1\}$ square-free. Then the {\it discriminant} $d_K$ of $K$ is given by 
\begin{align*}
d_K := \begin{cases} N, & \qquad \text{if } N \equiv 1 \pmod{4}, \\ 4N, & \qquad \text{otherweise}.\end{cases}
\end{align*}
Of course we have $K = \Q(\sqrt{d_K})$. One can show that 
\begin{align*}
\mathcal{O}_K = \begin{cases} \Z\left[ \sqrt{N}\right], & \qquad \text{if } N \not\equiv 1 \pmod{4}, \\ \Z\left[ \frac{1+\sqrt{N}}{2}\right], & \qquad \text{if } N \equiv 1 \pmod{4}. \end{cases}
\end{align*}
Next, we want to focus on the case of {\it imaginary quadratic fields}, i.e., $N < 0$. To any $\tau \in K \cap \HH$ we can assign a unique quadratic form $Q_\tau(x,y) := Ax^2 + Bxy + Cy^2$ with $\gcd(A,B,C) = 1$ ($Q_\tau$ is primitive) and $A > 0$ ($Q_\tau$ is positive definite), such that $Q_\tau(\tau,1) = A\tau^2 + B\tau + C = 0$. Note that $D(Q_\tau) := B^2 - 4AC$ is the discriminant of the quadratic form $Q_\tau$, and is therefore also called the discriminant $D(\tau)$ of $\tau$. Moreover, any quadratic form $Q_\tau$ can be mapped to an ideal class. In the case that $D(\tau) = d_K$, this is especially simple. 

\begin{proposition}[see \cite{Cox}, Theorem 5.30 on p. 101] \label{prop:Cox-prop-fundamentalDiscriminant} Let $K$ be an imaginary quadratic field of discriminant $d_K$. Then we have the following:
\begin{enumerate}
\item If $Q(x,y) = Ax^2 + Bxy + Cy^2$ is a positive definite quadratic form of discriminant $B^2 - 4AC = d_K$, then $\langle A, \frac{-B+\sqrt{d_K}}{2}\rangle_\Z$ is an ideal of $\mathcal{O}_K$. 
\item The map $\Xi_{d_K}$ sending $Q(x,y)$ to $\langle A, \frac{-B+\sqrt{d_K}}{2}\rangle_\Z$ induces an isomorphism of the form class group $C(d_K)$ and the ideal class group $\mathrm{Cl}_K$.
\end{enumerate}
\end{proposition}

With this proposition we can show that the assignment of an ideal class for $\tau$ is identical with respect to quadratic forms or its reduction type in the case of the fundamental discriminant. 

\begin{proposition} \label{prop:XiEqualsReductionFundamentalDiscriminant} Let $\tau = \frac{a}{b} \in K \cap \HH$ have discriminant $d_K$. Then $\left[\Xi_{d_K}(Q_\tau)\right] = \left[\left< a,b\right>\right]$. 
\end{proposition}

\begin{proof} By the remark after Proposition \ref{prop:ClassesOfFrations} the ideal class of interest does not depend on the choice of $a$ and $b$. We have
\begin{align*}
\tau = \frac{-B+\sqrt{d_K}}{2A}
\end{align*}
with $d_K = B^2 - 4AC$ and $Q_\tau(x,y) = Ax^2 + Bxy + Cy^2$, so by the first item of  Proposition \ref{prop:Cox-prop-fundamentalDiscriminant}
\begin{align*}
\left< A, \frac{-B+\sqrt{d_K}}{2}\right> = \left< A, \frac{-B+\sqrt{d_K}}{2}\right>_\Z = \Xi_{d_K}(Q_\tau),
\end{align*}
yielding the claim.
\end{proof}
We now generalize this result to arbitrary discriminants $D = f^2 d_K$, where $f = [\mathcal{O}_K : \mathcal{O}]$ is the conductor of the order $\mathcal{O}$ of discriminant $D$.

Proposition \ref{prop:Cox-prop-fundamentalDiscriminant} is in fact a special case of the statement for a general order $\mathcal{O}$.

\begin{proposition}[cf. \cite{Cox}, Theorem 7.7 on pp. 123--124] \label{prop:Cox-prop-non-fundamentalDiscriminant} Let $K$ be an imaginary quadratic field and $\mathcal{O}$ be an order in $K$ of discriminant $D = f^2d_K$, i.e., $\mathcal{O} = \left<1, f \omega_K\right>_\Z$. Then we have the following:
\begin{enumerate}
\item If $Q(x,y) = Ax^2 + Bxy + Cy^2$ is a positive definite quadratic form of discriminant $B^2 - 4AC = D$, then $\langle A, \frac{-B+\sqrt{D}}{2}\rangle_\Z$ is an invertible ideal of $\mathcal{O}$. Also, we have $\mathcal{O} = \langle 1, \frac{-B+\sqrt{D}}{2}\rangle_\Z$.
\item The map $\Xi_D$ sending $Q(x,y)$ to $\langle A, \frac{-B+\sqrt{D}}{2}\rangle_\Z$ induces an isomorphism of the form class group $C(D)$ and the ideal class group \[\mathrm{Cl}(\mathcal{O}) = \nicefrac{\left\{\mathfrak{a} \subset K: \mathfrak{a} \textup{ is an invertible fractional $\mathcal{O}$-ideal} \right\}}{\left\{\mathfrak{a} \subset K: \mathfrak{a} \textup{ is an invertible principal fractional $\mathcal{O}$-ideal} \right\}}\]
of $\mathcal{O}$.
\end{enumerate}
\end{proposition}

Further details about orders in quadratic imaginary fields may be found in \cite{Cox}*{\S 7}. Most importantly, we may still assign a specific class in $\mathrm{Cl}_K$ to a quadratic form via the next proposition.

\begin{proposition}[see \cite{Cox}, proof of Theorem 7.24 on pp. 132--134]
    Let $K$ be a number field and $\mathcal{O}$ an order in $K$ of discriminant $D$. The projection
    \[\mathrm{pr}_D: \mathrm{Cl}(\mathcal{O}) \longrightarrow \mathrm{Cl}_K, \quad [\mathfrak{a}] \mapsto [\mathfrak{a}\mathcal{O}_K]\]
    is well-defined.
\end{proposition}

\begin{proof}
    Let $\mathfrak{a}_1, \mathfrak{a}_2$ be two invertible $\mathcal{O}$-ideals lying in the same class of $\mathrm{Cl}(\mathcal{O})$. Then there is an invertible principal fractional $\mathcal{O}$-ideal $\mathfrak{b}$ such that $\mathfrak{a}_1\mathfrak{b} = \mathfrak{a}_2$. Since $\mathcal{O} \subset \mathcal{O}_K$, $\mathfrak{b}\mathcal{O}_K$ is a principal fractional ideal of $K$, so $[\mathfrak{a}_1\mathcal{O}_K] = [\mathfrak{a}_2\mathcal{O}_K]$ in $\mathrm{Cl}_K$.
\end{proof}

We now provide the following generalization of Proposition \ref{prop:XiEqualsReductionFundamentalDiscriminant}.

\begin{proposition}
    Let $\tau = \frac{a}{b} \in K \cap \HH$ have discriminant $D = f^2d_K$, then $\mathrm{pr}_D\left(\left[\Xi_{D}(Q_\tau)\right]\right) = \left[\left< a,b\right>\right]$.
\end{proposition}

\begin{proof}
    Generalizing Proposition \ref{prop:ClassesOfFrations} to arbitrary orders and using a similar argument as in Proposition \ref{prop:XiEqualsReductionFundamentalDiscriminant} due to Proposition \ref{prop:Cox-prop-non-fundamentalDiscriminant} (1), it follows that
    \[\mathrm{pr}_D\left(\left[\Xi_{D}(Q_\tau)\right]\right) = \mathrm{pr}_D\left(\left[\langle a, b \rangle_\mathcal{O} \right]\right) = \left[\left< a,b\right>\right]\]
    yielding the claim.
\end{proof}

In summary, we have seen that the map referred to in the remark after Proposition \ref{prop:InvarianceModTransf} harmonizes with the correspondence of quadratic forms with ideals. 

\begin{thebibliography}{99}


\bibitem{bru123}  J. H. Bruinier, G. van der Geer, G. Harder, D. Zagier, {\it The 1-2-3 of modular forms}, Universitext, Lectures from the Summer School on Modular Forms and their Applications held in Nordfjordeid, June 2004; Edited by Kristian Ranestad, Springer, Berlin, 2008.
   
\bibitem{Cohen} H. Cohen, {\it Number Theory, Volume I, Tools and Diophantine Equations}, Springer, 2007.

\bibitem{Cox} D. A. Cox, {\it Primes of the Form $x^2 + ny^2$}, Wiley, Second Edition, 2013.

\bibitem{Harman} G. Harman, {\it Metric Number Theory}, London Mathematical Society Monographs New Series, 1998.

\bibitem{Jarvis} F. Jarvis, {\it Algebraic Number Theory}, Springer, 2014.

\bibitem{Neukirch} J. Neukirch, {\it Algebraische Zahlentheorie}, Springer, 1995.

\bibitem{Palmer1} M. Palmer, \textit{The Duffin--Schaeffer Theorem in number fields}, Acta Arithmetica {\bf 196} (2020), 1-16.

\bibitem{Palmer2} M. Palmer, \textit{Vaaler's Theorem in number fields}, preprint. arXiv:2210.17064.


\bibitem{Tenenbaum} G. Tenenbaum, {\it Introduction to analytic and probabilistic number theory}, Graduate Studies in Mathematics, American Mathematical Society \textbf{163}, Third Edition, 2008. 




\end{thebibliography}
\end{document}